\documentclass[a4paper]{amsart}
\usepackage[latin9]{inputenc}
\usepackage{a4}
\usepackage{amsfonts}
\usepackage{amssymb}
\usepackage{amsmath}
\usepackage{amsthm}
\usepackage{changepage}
\usepackage{nomencl}
\usepackage{geometry}
\usepackage{soul}
\usepackage{tikz}
\usepackage{verbatim}
\usetikzlibrary{matrix}
\usepackage{hyperref}
\usepackage[toc,page]{appendix}
\usepackage{mathrsfs}
\usepackage[square, numbers, comma]{natbib}

\usepackage{enumerate}
\usepackage{amsbsy}

\usepackage[all]{xy}
\usepackage{pb-diagram,pb-xy}
   \bibpunct{[}{]}{,}{n}{}{,}
\input cyracc.def

\begin{document}
\providecommand{\keywords}[1]{\textbf{\textit{Keywords: }} #1}
\newtheorem{theorem}{Theorem}[section]
\newtheorem{lemma}[theorem]{Lemma}
\newtheorem{proposition}[theorem]{Proposition}
\newtheorem{corollary}[theorem]{Corollary}
\theoremstyle{definition}
\newtheorem{definition}{Definition}[section]
\theoremstyle{remark}
\newtheorem{remark}{Remark}
\newtheorem{conjecture}{Conjecture}
\newtheorem{question}{Question}

\def\p{\mathfrak{p}}
\def\q{\mathfrak{q}}
\def\s{\mathfrak{S}}
\def\Gal{\mathrm{Gal}}
\def\Ker{\mathrm{Ker}}
\def\Coker{\mathrm{Coker}}
\newcommand{\cc}{{\mathbb{C}}}   
\newcommand{\ff}{{\mathbb{F}}}  
\newcommand{\nn}{{\mathbb{N}}}   
\newcommand{\qq}{{\mathbb{Q}}}  
\newcommand{\rr}{{\mathbb{R}}}   
\newcommand{\zz}{{\mathbb{Z}}}

\title[Unramified extensions with certain perfect Galois groups]{Unramified extensions of quadratic number fields with Galois group $SL_2(7)$}
\author{Joachim K\"onig}
\address{Korea National University of Education, Department of Mathematics Education, 28173 Cheongju, South Korea}

\begin{abstract} We provide an infinite family of quadratic number fields with everywhere unramified Galois extensions of Galois group $SL_2(7)$. To my knowledge, this is the first instance of infinitely many such realizations for a perfect group which is not generated by involutions, a property which makes it difficult to approach for the problem in question and leads to somewhat delicate local-global problems in inverse Galois theory. \end{abstract}

\maketitle

\section{Introduction and main results}

This paper is concerned with the ``unramified inverse Galois problem", more concretely, the realization of certain prescribed groups as Galois groups of unramified Galois extensions over certain (small) number fields $K$. One reason for the interest in such extensions is their algebraic-topological relevance: the Galois groups of finite unramified Galois extensions of $K$ are exactly the finite quotients of the \'etale fundamental group $\pi_1^{et}(\textrm{Spec}(O_K))$ of the spectrum of its ring of integers (cf., e.g., \cite[Chapter X]{NSW}). 
Since $\mathbb{Q}$ is well-known to have no nontrivial unramified extensions by the Hermite-Minkowski theorem, the smallest fields of interest are quadratic number fields, and indeed, there is a wide variety of open problems here. Most famously, the Cohen-Lenstra heuristics on class groups \cite{CL84} predict the asymptotic distribution of quadratic number fields possessing an unramified extension with a given {\it abelian} Galois group.

For general finite groups $G$, the following is a folklore conjecture.
\begin{conjecture}
\label{conj:1}
For any finite group $G$, there exist infinitely many quadratic number fields $K$ such that $K$ possesses a Galois extension with Galois group $G$ unramified at all (finite and infinite) primes of $K$. \end{conjecture}

This is known to be true only in some special cases, including the case of cyclic groups (from class field theory, cf.\ \cite{AC55}), symmetric and alternating groups (e.g., \cite{Yam70} (which does not consider ramification at infinity) or \cite{Ked} (which does)), and a few other almost simple finite groups (e.g., \cite{KNS}). The common idea for the nonsolvable cases is to carefully specialize suitable function field extensions over $\mathbb{Q}$ (or equivalently, multiparameter polynomials over $\mathbb{Q}$) with prescribed Galois group to obtain extensions of $\mathbb{Q}$ with a desired ramification behavior. 
Further evidence for Conjecture \ref{conj:1} is provided by analogs over global function fields, which have proven fruitful and led to concrete distribution heuristics, in \cite{Wood}, for extensions $L/K$ with $[K:\mathbb{Q}]=2$ and $Gal(L/K)=G$. We will neither present nor use those heuristics in the present paper, although it is useful to keep in mind that the expected distribution differs depending on the concrete Galois group $Gal(\tilde{L}/\mathbb{Q}) \le G\wr C_2$, where $\tilde{L}$ denotes the Galois closure of $L/\mathbb{Q}$.

 Notably, for groups such as $G=S_n$, in order to solve Conjecture \ref{conj:1}, it suffices to obtain Galois extensions $L/\mathbb{Q}$ with group $S_n$ in which all nontrivial inertia groups are generated by transpositions,\footnote{Or more generally, by involutions, although allowing this does not seem to make the problem much easier.} and then base change to suitable quadratic number fields, killing ramification via Abhyankar's lemma; in the above terminology, this corresponds to the case $Gal(\tilde{L}/\mathbb{Q}) = G\times C_2$. 
The same  ``direct" approach can in principle succeed for all groups $G$ generated by involutions, but only for those. In particular, all non-abelian simple groups are generated by involutions (due to the Feit-Thompson theorem). Admittedly, in practice, realizations with involution inertia are not always easy to obtain; in particular, their existence still seems unknown for arbitrary alternating groups, whence those are instead dealt with via the above $S_n$-extensions $L/\mathbb{Q}$, in which $L/L^{A_n}$ is automatically unramified (see, e.g., \cite[Proposition 3.3.15]{JLY}).

In this paper, we wish to progress to nonsolvable groups which are {\it not} generated by involutions, i.e., for which the ``$G\times C_2$" approach cannot possibly be employed. Concretely, we show:

\begin{theorem}
\label{thm:main}
Let $G=SL_2(7)$.  
Then there exist infinitely many quadratic number fields possessing an everywhere unramified Galois extension with Galois group $G$.%
\end{theorem}

Note that $SL_2(7)$ is well-known to occur as Galois groups over quadratic number fields, and even over $\mathbb{Q}$, see, e.g., \cite{Swallow}.
% and \cite{Feit}. 
 However, the standard methods for its realization do not yield {\it unramified} extensions in any obvious way. To prove Theorem \ref{thm:main}, we instead ``manipulate", via suitable pullbacks, known function field extensions with Galois group $PGL_2(7)$, with the eventual goal of applying recent results (Theorem \ref{thm:kln}) to enforce certain prescribed behaviors of inertia and decomposition groups in specializations of these function field extensions. Once these are obtained, the final result can be deduced via (suitable) solutions to certain central embedding problems. Since there are explicit polynomials available for the ``starting" extensions in this process, we are eventually able to deduce an {\it explicit} infinite family of quadratic number fields fulfilling the assertion of the theorem, see Theorem \ref{thm:q7}.
 At the same time, it should be noted that the proof of Theorem \ref{thm:main} itself does not rely on the concrete shape of these polynomials, and can be carried out using only theoretical results.

Note also that the analog of Theorem \ref{thm:main} is known for the group $SL_2(3)$ (which, however, is a solvable group, namely a double cover of $A_4$) due to \cite{Serre2} (cf.\ also \cite[Proposition 7]{Yam}).

We also prove the following result applicable to infinitely many groups $SL_2(p)$, although with a weaker conclusion.
\begin{theorem}
\label{thm:main_weak}
Let $p$ be a prime with $p\equiv 3$ mod $4$ and $p\not\equiv -1$ mod $24$. Then there exist infinitely many quadratic number fields possessing an $SL_2(p)$-extension unramified everywhere outside of the prime ideal(s)\footnote{For what it's worth, our proof yields quadratic fields in which $p$ is inert, i.e., extended only by a single prime ideal.} dividing $p$.
\end{theorem}

We give some further motivation for the consideration of the groups $SL_2(p)$, as well as some insights about the difficulties in solving Conjecture \ref{conj:1} for them, in the next section.

\section{Some further motivation}

There are different classes of nonsolvable groups not generated by involutions, all interesting in their own right, and regarding unramified extensions of quadratic number fields, little seems to be known about any of them. Just to give one example, even though explicit families of quadratic number fields are known for each of $G=S_n$ ($n\ge 5$) and $H=C_m$ ($m\ge 3$), considering the same problem for the group $G\times H$ leads to delicate arithmetic questions (namely, about existence of discriminants falling into two classes simultaneously; see also Remark \ref{rem:dirprod}).

We focus instead on a different class of groups, namely {\it quasisimple} groups $G$ not generated by involutions. Here, a group $G$ is called quasisimple if it fits into an exact sequence $1\to Z\to G\to S\to 1$ with $Z=Z(G)$, $S$ a simple group, and $G$ is perfect (i.e., equal to its commutator subgroup). Perfectness is a logical assumption in order to make the problem of identifying unramified $G$-extensions ``disjoint" from the problem of class number divisibility of quadratic fields. Instead, several other challenges arise. Notably, in inverse Galois theory, Galois extensions with such groups are usually obtained via first realizing a suitable quotient group such as $S=G/Z(G)$ and then solving a central embedding problem. This alone is usually a nontrivial problem; see, for example, the realization of covering groups of the alternating and symmetric groups, which was investigated by many authors (\cite{Serre3}, \cite{Vila}, \cite{SchS}, \cite{Mestre} are just some of the relevant papers on this topic).

The additional condition that $G$ should not be generated by involutions means in particular that none of the involutions of $S$ lift to involutions in $G$, implying that ($|Z|$ is even and) the $2$-Sylow group of $G$ does not contain a subgroup isomorphic to $C_2\times C_2$, and hence must be cyclic or generalized quaternion. Since it is well-known that a nonabelian simple group cannot have a cyclic $2$-Sylow subgroup, we are left with the quaternion case. Here, the $2$-Sylow group of $S$ must be (the central quotient of a generalized quaternion group, and hence) dihedral, and it is then known (\cite[Theorem 2]{GW}) that $S$ is either $A_7$ or a group $PSL_2(q)$ with $q\ge 5$ an odd prime power.

Through this consideration (and via the known Schur multipliers of the simple groups above), we have arrived naturally at the set of groups to be considered: they are the double covers $SL_2(q) = 2.PSL_2(q)$ (with $q$ odd) as well as the even degree covering groups $2.A_7$, $6.A_7$ and $6.PSL(2,9) \cong 6.A_6$.

The joint idea to attack these groups for our problem is the fact that they embed as index-$2$ normal subgroups into groups which {\it are} generated by involutions (namely, central extensions of $\Gamma = PGL_2(q)$ and of $\Gamma =S_7$ by $C_2$ (and in two cases, noncentral extensions by $C_6$), respectively). This is analogous to the embedding of $A_n$ into $S_n$ as discussed above, but additional problems arise here. 
Namely, we additionally need to begin with a prescribed ramification behavior in the $\Gamma$-extension, and make sure that the embedding problem with kernel $C_2$ is solvable without introducing extra ramification. We use a function field approach to control the behavior of inertia and decomposition groups under specialization. In this setting, it turns out that the embedding problem usually leads to simultaneous requirements on the ramification index as well as the residue extension at given branch points of the function field extension. E.g., recall once more that the problem for the simple group $A_n$ itself is solved via constructing $S_n$-extensions of $\mathbb{Q}$ in which all non-trivial inertia groups are generated by transpositions, which is achieved via finding degree-$n$ extensions with squarefree discriminant, whose existence is well-known. However in order to ensure that such an $S_n$-extension has no local obstruction to embedding into an extension with Galois group $2.S_n$, one needs to impose conditions on the Frobenius class of (all!) prime divisors of the squarefree values of such discriminants. This is usually difficult to guarantee without relying on open conjectures (cf., e.g., \cite[Section 4]{Ked} for some informal remarks on the difficulty of such problems).
%, and in a setting of specialization of function field extensions amounts to finding values of certain (discriminant) polynomials with all prime divisors in some prescribed Chebotarev set. In this light, Theorem \ref{thm:main} means that in some special cases, these extra requirements can be obtained unconditionally.

\section{Prerequisites}

\subsection{Local behavior of specialization of function field extensions}
We review some known results about the behaviour of inertia and decomposition groups in specializations of function field extensions.
Since these occur in several predecessor works, we restrict to the most directly relevant results. For further background on function field extensions and their specialization, cf., e.g., \cite{KLN}, Sections 2 and 3.

For a finite Galois extension $N/K(t)$ of function fields and any value $t_0\in \mathbb{P}^1(K)$, we denote by $N_{t_0}/K$ the specialization of $N/K(t)$ at $t\mapsto t_0$, i.e., the residue extension at any place extending $t\mapsto t_0$ in $N$. Recall that $N/K(t)$ is called $K$-regular if $K$ is algebraically closed in $N$. 
The following well-known theorem, cf.\ \cite[Prop.\ 4.2]{Beck}, ties the ramification in specializations of a function field extension to the inertia groups at branch points of that function field extension.

\begin{theorem}
\label{thm:beck}
Let $K$ be a number field and $N/K(t)$ a finite $K$-regular Galois extension with Galois group $G$.
Then there exists a finite set $\mathcal{S}_0$ of primes, depending only on $N/K(t)$, such that the following holds for every prime $\mathfrak{p}$ of $K$ not in $\mathcal{S}_0$:\\
If $t_0\in K$ is not a branch point of $N/K(t)$ then the following condition is necessary for $\mathfrak{p}$ to be ramified in the specialization $N_{t_0}/K$:
 $$e_i:=I_{\mathfrak{p}}(t_0, t_i)>0 \text{ for some (automatically unique, up to algebraic conjugates) branch point $t_i$.}$$
 Here $I_{\mathfrak{p}}(t_0,t_i)$ is the intersection multiplicity of $t_0$ and $t_i$ at the prime $\mathfrak{p}$.
Furthermore, the inertia group of a prime extending $\mathfrak{p}$ in the specialization $N_{t_0}/K$ is then conjugate in $G$ to $\langle\tau^{e_i}\rangle$, where $\tau$ is a generator of an inertia subgroup over the branch point $t\mapsto t_i$ of $K(t)$.
\end{theorem}

Regarding the definition of intersection multiplicity $I_{\mathfrak{p}}$ occurring in Theorem \ref{thm:beck}, note that in the special case $K=\mathbb{Q}$ this may be defined conveniently in the following way: Let $f(X)\in \zz[X]$ be the irreducible polynomial of $t_i$ over $\zz$ and $\tilde{f}(X,Y)$ its homogenization (with $\tilde{f}:=Y$ in the special case $t_i=\infty)$. Let $t_0=\frac{a}{b}$ with $a,b\in \zz$ coprime, and let $p$ be a prime number. Then $I_p(t_0,t_i)$ is the multiplicity of $p$ in $\tilde{f}(a,b)$.

The following result (cf.\ \cite[Thm.\ 4.1]{KLN}) clarifies the structure of decomposition groups in specializations.
To state it, we need to introduce a bit more notation. Given a finite Galois extension $N/K(t)$ with group $G$ and a $K$-rational place $t\mapsto a\in \mathbb{P}^1(K)$, denote by $D_a$ and $I_a$ the decomposition and inertia group at the place $t\mapsto a$ respectively. Furthermore, given a prime $p$ of $K$, denote by $D_{a,p}$ the decomposition group at $p$ in the residue extension $N_a/K$. Note that the Galois group of $N_a/K$ is identified with $D_a/I_a$, and hence $D_{a,p}$ is identified with a subgroup of this group.

\begin{theorem}
\label{thm:kln}
Let $K$ be a number field and $N/K(t)$ a finite $K$-regular Galois extension with Galois group $G$.
Let $t\mapsto t_i \in \mathbb{P}^1(\overline{K})$ be a branch point of $N/K(t)$, and let $I_{t_i}$ and $D_{t_i}$ denote the inertia and decomposition group at $t\mapsto t_i$ in $N(t_i)/K(t_i)(t)$.
Then there exists a finite set $\mathcal{S}_1$ of primes of $K$, depending only on $N/K(t)$ and containing the set $\mathcal{S}_0$ from Theorem \ref{thm:beck}, such that for all primes $\mathfrak{p}$ of $K$ not in $\mathcal{S}_1$ and all non-branch points $a\in \mathbb{P}^1(K)$ of $N/K(t)$, the following hold:

\begin{itemize}
\item[a)] Assume that $\mathfrak{p}$ ramifies in $N_a/K$ and $I_{\mathfrak{p}}(a,t_i)>0$.\footnote{The latter is guaranteed to happen for some branch point $t_i$, unique up to algebraic conjugates, by Theorem \ref{thm:beck}.} 
Let $D_{t_i,\mathfrak{p}'}$ denote the decomposition group of $N(t_i)_{t_i}/K(t_i)$ at the (unique) %
 prime $\mathfrak{p}'$ of $K(t_i)$ extending $\mathfrak{p} $ such that $I_{\mathfrak{p}'}(a,t_i)>0$. Then the decomposition group $D_{a,\mathfrak{p}}$ at $\mathfrak{p}$ in $N_a/K$ is identified (up to conjugacy in $G$) with a subgroup of $D_{t_i}$, and fulfills $\varphi(D_{a,\mathfrak{p}}) = D_{t_i,\mathfrak{p}'}$, where $\varphi:D_{t_i}\to D_{t_i}/I_{t_i}$ is the canonical epimorphism.
\item[b)] In particular, if additionally $I_{\mathfrak{p}}(a,t_i) = 1$, then $D_{a,\mathfrak{p}} = \varphi^{-1}(D_{t_i,\mathfrak{p}'})$.
\end{itemize}
\end{theorem}

\begin{remark}
\label{rem:explicit}
\begin{itemize}
\item[a)]
The finite exceptional set $\mathcal{S}_1$ in Theorem \ref{thm:kln} can be made explicit, and this is sometimes relevant when doing calculations with concrete extensions, as we will do in Theorem \ref{thm:q7}. In order to avoid too much technicality, we only refer to \cite[Theorem 2.2]{KN21} for an effective version of Theorem \ref{thm:kln}.
\item[b)] For many applications, it is enough to obtain {\it some} specialization values $a\in \mathbb{P}^1(K)$ yielding a certain prescribed local behavior, rather than controlling the behavior for all specialization values. In this context, the notion of a {\it universally ramified prime} (see \cite{BSS} or \cite{KNS}) becomes useful: let $\mathcal{U}(N/K(t))$ be the set of all primes of $K$ ramifying in all specializations $N_a/K$ where $a\in \mathbb{P}^1(K)$ is a non-branch point of $N/K(t)$. Let $\mathcal{S}$ be any finite set of primes of $K$ disjoint from $\mathcal{U}(N/K(t))$. Then there exists a non-empty and $\mathcal{S}$-adically open set of values $a\in \mathbb{P}^1(K)$ such that $N_a/K$ is unramified at all $p\in \mathcal{S}$ (an easy consequence of Krasner's lemma). In particular, by choosing $\mathcal{S}:=\mathcal{S}_1 \setminus \mathcal{U}(N/K(t))$ with $\mathcal{S}_1$ as in Theorem \ref{thm:kln}, we see that for all such values $a$, the assertion of Theorem \ref{thm:kln} actually holds with $\mathcal{S}_1$ replaced by $\mathcal{U}(N/K(t))$, and the latter set is often much more accessible and can be bounded very effectively by simply comparing some very few specializations.
\end{itemize}
\end{remark}

We will also require the following lemma, describing the behavior of residue extensions under certain rational pullbacks.
\begin{lemma}
\label{lem:pb}
Let $N/K(t)$ be a $K$-regular Galois extension with group $G$, let $I$ and $D$ denote the inertia and decomposition group at the place  $t\mapsto 0$, and assume that the extension $1\to I \to D \to D/I \to 1$ splits. Then there exists $\alpha\in K^\times$ such the following holds:

If $e:=|I|$ and $u$ is such that $u^e = \alpha t$, then the extension $N(u)/K(u)$ is unramified at $u\mapsto 0$ with decomposition group isomorphic to $D/I$.
\end{lemma}
\begin{proof}
Let $U$ be a complement to $I$ in $D$, let $\widehat{N}/K((t))$ denote the completion of $N/K(t)$ at $t\mapsto 0$, and $F/K$ the residue extension of $N/K(t)$ at $t\mapsto 0$. In particular, $Gal(\widehat{N}/K((t))) \cong D$, and $\zeta_e \in F$. %Furthermore, the maximal unramified extension in $\widehat{N}/K((t))$ (of group isomorphic to $U$)
Let $\widehat{N}^U$ denote the fixed field of $U$. Then $\widehat{N}^U/K((t))$ is totally ramified of degree $e$, and hence equal to $K((\sqrt[e]{\alpha t})) =:K((u))$ for a suitable $\alpha\in K$. Consequently, $\widehat{N} = F((u))$. Now firstly, since $K(u)/K(t)$ is totally ramified of index $e$ at $t\mapsto 0$, the extension  $E(u)/K(u)$ is unramified at $u\mapsto 0$. Furthermore, the residue extension $E(u)_0/K$ is necessarily contained in the compositum of $K(\zeta_e)((u))$ (i.e., the Galois closure of $K((u))/K((t))$) and $\widehat{N} = F((u))$; since $\zeta_e\in F$, this compositum is $F((u))$, and thus its residue extension is $F$. Conversely, the completion at $u\mapsto 0$ has to be of degree at least $|D|/[K(u):K(t)] = |D/I| = [F:K]$, yielding the assertion.
\end{proof}

\begin{remark}
The splitting condition in the assumption of Lemma \ref{lem:pb} is necessary. Indeed, an analog of Theorem \ref{thm:kln} (with the base field $K$ replaced by the function field $K(u)$) shows that the decomposition group at $u\mapsto 0$ in $N(u)/K(u)$ is a subgroup of $D$ projecting onto $D/I$. This observation is essentially the reason why we are so far unable to include the (smaller) group $G=SL_2(5)$ in Theorem \ref{thm:main}.
\end{remark}

\subsection{Embedding problems}
\label{sec:embed_basics}
We recall some basic terminology around embedding problems; cf., e.g., \cite[Chapter III.5]{NSW}.  
A {\it finite embedding problem} over a field $K$ is a pair $(\varphi:G_K\to G,\varepsilon:\tilde{G} \to G)$, where $\varphi$ is a (continuous) epimorphism from the absolute Galois group $G_K$ of $K$ onto $G$, 
and $\varepsilon$ is an epimorphism between finite groups $\tilde{G}$ and $G$ fitting in an exact sequence $1\to N\to \tilde{G}\to G\to 1$. The kernel $N=\ker(\varepsilon)$ is called the kernel of the embedding problem. An embedding problem is called {\it central} if $\ker(\varepsilon) \le Z(\tilde{G})$. %, 
A (continuous) homomorphism $\psi:G_K\to \tilde{G}$ is called a \textit{solution} to $(\varphi, \varepsilon)$ if the composition  $\varepsilon \circ \psi$ equals $\varphi$. 
In this case, the fixed field of $\ker(\psi)$ is called a {\it solution field} to the embedding problem.
A solution $\psi$ is called a \textit{proper solution} if it is surjective. In this case, the field extension of the solution field over $K$ has full Galois group $\tilde{G}$. 

If $K$ is a number field and $\mathfrak{p}$ is a prime of $K$, every embedding problem $(\varphi, \varepsilon)$ induces an associated {\it local embedding problem} $(\varphi_{\mathfrak{p}}, \varepsilon_{\mathfrak{p}})$ defined as follows: $\varphi_{\mathfrak{p}}$ is the restriction of $\varphi$ to $G_{K_{\mathfrak{p}}}$ (well defined up to fixing an embedding of $\overline{K}$ into $\overline{K_{\mathfrak{p}}}$), and $\varepsilon_{\mathfrak{p}}$ is the restriction of $\varepsilon$ to $\varepsilon^{-1}(G({\mathfrak{p}}))$, where $G(\mathfrak{p}) := \varphi_{\mathfrak{p}}(G_{K_{\mathfrak{p}}})$.

We require two well-known results on central embedding problems.
The first  (cf.\ \cite[Chapter IV, Cor.\ 10.2]{MM}) is a local global-principle for the solvability of such embedding problems, which is essentially a consequence of the local-global principle for Brauer embedding problems over number fields (\cite[Chapter IV, Cor.\ 7.8]{MM}).

\begin{proposition}
\label{lem:prime_kernel}
Let $\Gamma = C.G$ be a central extension of $G$ by a cyclic group $C$ of prime order and $\varepsilon:\Gamma\to G$ the canonical projection.
Let $\varphi:G_{\mathbb{Q}}\to G$ be a continuous epimorphism. Then the embedding problem 
 $(\varphi, \varepsilon)$ is solvable if and only if all associated local embedding problems $(\varphi_p, \varepsilon_p)$ are solvable, where $p$ runs through all primes of $\mathbb{Q}$ (including the infinite one).
\end{proposition}

Since the local embedding problem at an unramified prime is always solvable, the above lemma gives an efficient method to check for global solvability of a given embedding problem by investigating only the finitely many ramified primes of $K/\mathbb{Q}$.

The second result (see \cite[Prop.\ 2.1.7]{Serre}) is a tool to control ramification in solutions of central embedding problems, assuming solvability. \begin{proposition}
\label{prop:serre}
Let $\Gamma = C.G$ be a central extension of $G$ by a finite abelian group $C$, let $\varepsilon:\Gamma\to G$ be the canonical projection, and let $\varphi: G_{\mathbb{Q}}\to G$ be a continuous epimorphism such that the embedding problem $(\varphi,\varepsilon)$ has a solution. 
For each finite prime $p$, let $\tilde{\varphi_p}: G_{\mathbb{Q}_p}\to \Gamma$ be a solution of the associated local embedding problem $(\varphi_p,\epsilon_p)$, 
 chosen such that all but finitely many $\tilde{\varphi_p}$ are unramified. Then there exists a (not necessarily proper) solution $\tilde{\varphi}: G_{\mathbb{Q}}\to \Gamma$ of $(\varphi,\epsilon)$ such that for all finite primes $p$, the restrictions of $\tilde{\varphi}$ and $\tilde{\varphi_p}$ to the inertia group inside $G_{\mathbb{Q}_p}$ coincide. In particular, $\tilde{\varphi}$ is ramified exactly at those finite primes $p$ for which $\tilde{\varphi_p}$ is ramified.
\end{proposition}

\section{Proof of Theorem \ref{thm:main_weak}}
\begin{proof}
We first describe the general strategy. We will show the assertion by first constructing ``suitable" $PGL_2(p)$-extensions of $\mathbb{Q}$ and then embedding them into $2.PGL_2(p)$-extensions using Propositions \ref{lem:prime_kernel} and \ref{prop:serre}. Here, by $2.PGL_2(p)$, we mean the stem extension of $PGL_2(p)$ in which the ``outer" involutions (i.e., the one not contained in $PSL_2(p)$) split. By ``suitable" extensions, we mean extensions in which the intermediate extension of group $PSL_2(p)$ is unramified everywhere away from $p$, meaning that all nontrivial inertia groups other than possibly over $p$ must be generated by involutions outside of $PSL_2(p)$. There is only one class of such involutions in $PGL_2(p)$. The decomposition group $D$ at a prime with such an inertia subgroup must of course be contained in the centralizer of such an involution; this centralizer contains (e.g.) a subgroup $C_2\times C_2$, which is non-split in $2.PGL_2(p)$ and whose preimage is isomorphic to the dihedral group $D_4$. Thus, if in the given $PGL_2(p)$-extension, some prime $q$ had this decomposition and inertia subgroup, then assuming the existence of a lifting to $2.PGL_2(p)$, the decomposition group at $q$ would be $D_4$. Since however the ``outer" involution lifts to a (non-central) element of order $2$ in $D_4$, this would imply that the inertia subgroup in the $2.PGL_2(p)$-extension would be a non-normal of the decomposition group, a contradiction. In other words, a prime with decomposition group $C_2\times C_2$ would constitute a local obstruction to the solvability of the central embedding problem. In our construction below, we thus need to ensure among others (using Theorem \ref{thm:kln}) that such primes do not occur. As we will see, for primes with inertia group generated by an outer involution, the above is in fact the only local obstruction.

Due to the assumption $p\equiv 3$ mod $4$ and $p\not\equiv -1$ mod $24$, at least one of $2$ and $3$ is a non-square modulo $p$. Thus, by \cite[Chapter I, Cor.\ 8.10]{MM}, there exists a $\mathbb{Q}$-regular Galois extension $E/\mathbb{Q}(t)$ with Galois group $PGL_2(p)$, and with the following properties: $E$ is ramified at exaxtly three points $t_i\in \mathbb{P}^1(\overline{Q})$ ($i=1,2,3$), all $\mathbb{Q}$-rational, and the corresponding inertia groups are generated by elements of conjugacy class $pA$ (for $t_1)$, $4A$ or $6A$ (for $t_2$, depending on whether $2$ or $3$ is a non-square modulo $p$), and $2B$ (for $t_3$), in Atlas notation. Here $2B$ is the unique class of involutions outside of $PSL_2(p)$. Up to a fractional linear transformation in $t$, we may and will assume $t_1=0$, $t_2=\infty$ and $t_3=1$. Now consider the residue extensions at the branch points $t_1$ and $t_3$. 
Since $p\equiv 3$ mod $4$, the elements of class $2B$ are involutions with exactly two fixed points in the standard action of $PGL_2(p)$ on $p+1$ points. The centralizer of an element of class $2B$ in $PGL_2(p)$ is isomorphic to $C_2\times D_{(p-1)/2}$, where $D_{(p-1)/2}$ denotes the dihedral group of order $p-1$, and thus the residue extension at $t_1$ is of degree dividing $p-1$, with Galois group a subgroup of $D_{(p-1)/2}$. Assume that it is indeed of even degree. Then it has a unique quadratic subextension, which we denote by $\mathbb{Q}(\sqrt{\beta})/\mathbb{Q}$. Consider the extension $E(s)/\mathbb{Q}(s)$, where $s$ is such that $\beta t = s^2$. This is still a $\mathbb{Q}$-regular $PGL_2(p)$-extension (indeed, the quadratic extension $\mathbb{Q}(s)/\mathbb{Q}(t)$ has only one branch point in common with the quadratic subextension of $E/\mathbb{Q}(t)$), ramified at $s=0$, $s=\infty$ and $s=\pm\sqrt{\beta}$. In particular, the residue field at the latter two branch points ``swallows" the above quadratic subextension, whence the residue {\it extension} at any of those two branch points in $E(s)/\mathbb{Q}(s)$ is now cyclic of odd degree.

Now consider the residue extension at $t_1$.
Since the normalizer of a subgroup of order $p$ in $PGL_2(p)$ is isomorphic to $C_p\rtimes C_{p-1}$,  the residue extension at this branch point is of degree dividing $p-1$; and since furthermore, the residue extension at a branch point of ramification index $e$ must always contain the $e$-th roots of unity (e.g., \cite[Lemma 2.3]{KLN}), it follows that this residue extension in fact equals $\mathbb{Q}(\zeta_{p})/\mathbb{Q}$. Since $s\mapsto 0$ is the unique place over $t\mapsto 0$ in $E(s)$, the same holds for the residue extension of $E(s)/\mathbb{Q}(s)$ at $s=0$. It then follows from Lemma \ref{lem:pb} 
that for suitable $\alpha\in \mathbb{Q}$, the extension $E(u)/\mathbb{Q}(u)$, where $u$ is such that $s=\alpha u^p$, is $\mathbb{Q}$-regular with group $PGL_2(p)$, unramified at $u=0$ and with residue field extension still equal to $\mathbb{Q}(\zeta_p)/\mathbb{Q}$ (in particular, unramified outside of $p$ and $\infty$). 
Furthermore, at the branch points of $E(u)/\mathbb{Q}(u)$ extending $s=\pm \sqrt{\beta}$, the Galois group of the residue extension is guaranteed to inject into the one at $s=\pm \sqrt{\beta}$. This is because $\mathbb{Q}(u)/\mathbb{Q}(s)$ is unramified at those points, whence the inertia group $I$ at such a branch point in $E(u)/\mathbb{Q}(u)$ is the same as the corresponding one in $E(s)/\mathbb{Q}(s)$, and the decomposition group $D$ certainly cannot grow under pullback, implying the quotient of decomposition over inertia group cannot grow either.

Now from Theorem \ref{thm:kln}, we obtain the following: For all $u_0\in \mathbb{Q}$ of the form $u_0=\frac{u_1}{u_2^{6}}$ (with integers $u_1,u_2$)\footnote{The power in the denominator indeed only needs to be a multiple of the ramification index at $\infty$ in $E(u)/\mathbb{Q}(u)$, in order to avoid ramification induced by this branch point; by construction, this ramification index is either $2$ or $3$, depending on which of the two possible ramification types we started with.} and all but finitely many primes $q$, if the specialization $E(u)_{u_0}/\mathbb{Q}$ is ramified at $q$, then its inertia group in generated by an element of class $2B$ and its decomposition group embeds into $C_2\times C_{(p-1)/2} \equiv C_{p-1}$. In particular, no such prime yields a local obstruction to the embedding problem induced by $2.PGL_2(p)\to PGL_2(p)$, since the group elements of order $2B$, and a fortiori the odd order elements of $C_{(p-1)/2}$ are split in the central extension. To deal with the remaining finitely many primes $q$, it suffices to additionally choose $u_0$ $q$-adically sufficiently close to $u=0$. Indeed, since $E(u)_0/\mathbb{Q} = \mathbb{Q}(\zeta_p)/\mathbb{Q}$ is unramified away from $\{p,\infty\}$, Remark \ref{rem:explicit}b) yields that for $q\ne p$ finite, the extension $E(u)_{u_0}$ is then unramified at $q$; whereas for $q=p$, it is cyclic and totally ramified of degree $p-1$. Again, as above, the cyclic subgroups of order $p-1$ are split in the central extension. Finally, the infinite prime yields no local obstruction either, since specializing sufficiently close to $u=0$ (in the archimedean absolute value)\footnote{Note that this is possible within the set of values $u_0$ chosen above.} yields as complex conjugation an involution inside $C_{p-1}$, i.e., once again an element of class $2B$ in $PGL_2(p)$.

Via Proposition \ref{lem:prime_kernel}, we may thus solve the embedding problem induced by $G_\mathbb{Q} \to Gal(E(u)_{u_0}/\mathbb{Q})$ and $2.PGL_2(p)\to PGL_2(p)$, and due to Proposition \ref{prop:serre}, this can in fact be achieved without any ramification inside the kernel (since all non-trivial inertia groups were split). Thus, if $K$ denotes the fixed field of $SL_2(p)$ in the resulting extension $L/\mathbb{Q}$, then all inertia groups in $Gal(L/\mathbb{Q})$, except the ones at $p$, are generated by involutions outside of $Gal(L/K)$, whence $L/K$ is ramified only above $p$. This concludes the proof.
\end{proof}

\begin{remark}
The extensions of $\mathbb{Q}$ with Galois group $PGL_2(p)$ constructed in the above proof are in particular {\it locally cyclic}, i.e., all decomposition groups at primes of $\mathbb{Q}$ are cyclic. This property was previously studied in \cite{KK21}, and locally cyclic extensions were obtained, e.g., for the special case $G=PGL_2(7)$ (\cite[Theorem 5.5]{KK21}). Our above result (for infinitely many primes $p$) in particular improves on \cite[Theorem 2.1]{KK21}. 
\end{remark}

\section{Proof of Theorem \ref{thm:main}} 

We first prove \ref{thm:main} in a purely theoretical way, using only ramification data and group-theoretical properties. Afterwards, we will make the conclusion more concise using computations with some explicit polynomials.

%\subsection{The case $G=SL_2(7)$}

\begin{proof}
In the following, we adapt the proof of Theorem \ref{thm:main_weak} somewhat for the special case $p=7$, in order to get rid of ramification at $p$. Note that since $3$ is a non-square modulo $7$, the extensions $E/\mathbb{Q}(t)$ used in the proof of Theorem \ref{thm:main_weak} are of ramification type $(pA, 6A, 2B) = (7A, 6A, 2B)$. After pullback by an extension $\mathbb{Q}(u)/\mathbb{Q}(t)$ of rational function fields as in the above proof, we obtain a $PGL_2(7)$-extension $E(u)/\mathbb{Q}(u)$ with the following ramification data: inertia group of order $3$ at $u=\infty$, and a total of $14$ branch points of ramification index $2$ and with residue extension of odd degree; furthermore, no ramification above $u\mapsto 0$. We specialize $u\mapsto u_0$ essentially as above, except that we will choose $u_0$ $7$-adically close to $u=\infty$ rather than to $u=0$. Note that the normalizer of a subgroup of order $3$ in $PGL_2(7)$ is isomorphic to $S_3\times C_2 \cong C_3 \rtimes (C_2\times C_2)$, whence the residue extension $E(u)_{\infty}/\mathbb{Q}$ is of degree dividing $4$, and on the other hand (due to the ramification index being $3$) contains $\mathbb{Q}(\zeta_3)$. 
Note furthermore that $\mathbb{Q}_7$ conains the third roots of unity, whence the completion of $E(u)\cdot \mathbb{Q}_7/\mathbb{Q}_7(u)$ at $u\mapsto \infty$ has Galois group embedding into the centralizer $C_{PGL_2(7)}(C_3)$,\footnote{Here we have used the general fact that the action of the decomposition group on an inertia subgroup of order $e$ is via the Galois action of the $e$-th cyclotomic extension; see, e.g., \cite[Lemma 2.1(1)]{KN21}.} and in particular the residue extension $E(u)_{\infty} \cdot \mathbb{Q}_7 / \mathbb{Q}_7$ is either trivial or generated by an involution centralizing $C_3$. Such an involution necessarily lies outside of $PSL_2(7)$. 

As before, Lemma \ref{lem:pb} yields that there exists $\gamma\in \mathbb{Q}$ such that for $v$ defined via $u=\gamma v^3$, the extension $E(v)/\mathbb{Q}(v)$ is still $\mathbb{Q}$-regular of group $PGL_2(7)$, unramified at $v=\infty$ and with residue extension $E(v)_\infty/\mathbb{Q}$ contained in the above residue extension $E(u)_\infty/\mathbb{Q}$.

Specialize $v\mapsto v_0$ $7$-adically close to $v=\infty$. Then by Krasner's lemma, the local behavior of $E(v)_{v_0}/\mathbb{Q}$ at $p=7$ is the same as that of $E(v)_\infty/\mathbb{Q}$. In particular, by the above, the completion $E(v)_0\cdot \mathbb{Q}_7/\mathbb{Q}_7$ is either trivial or quadratic with Galois group generated by an involution of class $2B$. As before, this constitutes no local obstruction to the central embedding problem with kernel $C_2$, and we thus obtain (for $p=7$) the conclusion of Theorem \ref{thm:main_weak} {\it without} the exception at $p$.
\end{proof}

\begin{remark}
Theorem \ref{thm:main_weak} and its proof explicitly include also the case $p=3$. We leave it as an exercise to the reader to adapt the argument in the proof of Theorem \ref{thm:main} above to regain the analogous (and already known) statement for $p=3$ (i.e., infinitely many quadratic number fields with {\it everywhere} unramified $SL_2(3)$-extensions). 
\end{remark}

In the following, we make the assertion of Theorem \ref{thm:main} more precise by giving a fully explicit family of quadratic number fields possessing unramified $SL_2(7)$-extensions. Theorem \ref{thm:q7} should be compared with previous results on explicit families of number fields with class number divisible by prescribed integers $n$ (see, e.g., \cite{CHKP} for an overview as well as new results).

\begin{theorem}
\label{thm:q7}
For all 
odd integers $t\in 2\mathbb{Z}+1$,  the quadratic number field 
 $\mathbb{Q}(\sqrt{-7\cdot (108 + 7^{10}t^6)})$ possesses an everywhere unramified $SL_2(7)$-extension. \end{theorem}
\begin{proof}
We consider the splitting field $E$ of the degree-$8$ polynomial $f(t,X) = 7^3X^7(X+7) - t(X^2+X+7)$ over $\mathbb{Q}(t)$. This has Galois group $G=PGL_2(7)$ over $\mathbb{Q}(t)$ (e.g., it is, up to multiplication of the variables by suitable constants, the special case $a=0$ of the polynomial in \cite[Theorem 5.1]{Malle}). Furthermore, $E/\mathbb{Q}(t)$ is ramified exactly  over $t\mapsto 0$, $t\mapsto \infty$ and one more rational point $t\mapsto t_0:=-7^{10}/(2^2\cdot 3^3)$, with inertia group generator of cycle structure $(7.1)$, $(6.1^2)$ and $(2^3.1^2)$ respectively. Following the above proofs, we manipulate $f$ in such a way that all ramification in specializations will be ``decided" by the branch point of ramification index $2$. We first find out about the residue extension at this branch point. Firstly, the decomposition group at this branch point needs to be contained in the centralizer of an element of cycle structure $(2^3.1^2)$ in $PGL_2(7)$, which is of order $12$. Thus the residue extension at $t\mapsto t_0$ is a Galois extension of group embedding into $C_{G}(I)/I \cong S_3$, and from factorizing $f(t_0,X)$ (which is the product of a quadratic polynomial and the square of a cubic polynomial), we find that it is indeed an $S_3$-extension, with quadratic subextension $\mathbb{Q}(\sqrt{-3})/\mathbb{Q}$. This happens to be the same quadratic extension as $\mathbb{Q}(\sqrt{t_0})/\mathbb{Q}$. Thus, setting $t(s):=1/s^2$, the splitting field $E(s)$ of $f(t(s),X)$ over $\mathbb{Q}(s)$ is a $PGL_2(7)$-extension ramified exactly over $s\in \{0,\infty, \pm \sqrt{t_0}\}$ (with ramification index $3$ at $s\mapsto 0$ due to Abhyankar's lemma, since this place extends $t\mapsto \infty$ of ramification index $2$), and now, the residue extension at $s\mapsto \pm \sqrt{1/t_0}$ is a $C_3$-extension of the base field $\mathbb{Q}(\sqrt{t_0})$. Pulling back further by setting $s(u):=u^3$, we obtain a $PGL_2(7)$=extension $E(u)/\mathbb{Q}(u)$ ramified only at the $6$-th roots of $1/t_0$, as well as $u\mapsto \infty$ (and in particular, unramified at $u\mapsto 0$ due to Abhyankar's lemma). Since the residue field $\mathbb{Q}(\sqrt[6]{t_0})$ contains $\mathbb{Q}(\sqrt{t_0}) = \mathbb{Q}(\sqrt{-3})$, it still remains true that the residue extension of $E(u)/\mathbb{Q}(u)$ at any of these finite branch points is a $C_3$-extension. Theorem \ref{thm:kln} then gives the following: There exists a finite set  $\mathcal{S}$ of primes such that, for any integral specialization $u\mapsto u_0\in \mathbb{Z}$ and any prime $p$ not contained in $\mathcal{S}$, if $p$ ramifies in $E(u)_{u_0}/\mathbb{Q}$, then its inertia group must be generated by an involution of cycle type $(2^3.1^2)$, and its decomposition group must be either $C_2$ or $C_2\times C_3$. Since both these groups split in $2.PGL_2(7)$, the corresponding central embedding problem then has no local obstruction coming from any finite prime $p$ outside $\mathcal{S}$. We now verify computationally that in fact the primes in the exceptional set $\mathcal{S}$ pose no obstruction either.
This is done by specializing $f(t(s(u)),X)$, which up to constant factor equals 
$$F(u,X) = (7u^2)^3X^7(X+7)-(X^2+X+7),$$ at odd integer values $u\mapsto u_0 \in 2\mathbb{Z}+1$. 
Calculating the discriminant of $F(u,X)$ with respect to $X$ and using \cite[Theorem 2.2]{KN21}, as indicated in Remark \ref{rem:explicit}a), one verifies that one may in fact take $\mathcal{S} = \{2,3,7, \infty\}$. Firstly note that $F(u_0,X)$ mod $2$ is separable for odd $u_0$, whence $2$ is unramified in all such specializations and thus poses no problem. Neither does the infinite prime, since our $PGL_2(7)$-extensions are such that the quadratic subextension is imaginary, i.e., complex conjugation is given by an involution of cycle structure $(2^3.1^2)$. For the remaining primes $p=3$ and $7$, it may be directly verified, aided by Magma, 
 that all specializations which ramify at $p$ (in the case of $p=7$, these are actually all specializations $u\mapsto u_0\in 2\mathbb{Z}+1$; in the case of $p=3$, ramification happens only when $3|u_0$) have cyclic decomposition group generated by an involution of cycle type $(2^3.1^2)$.\footnote{Note that, in order to strictly verify this claim, one of course need not do calculations for infinitely many specializations, since the behavior of the decomposition group is the same for all values $u_0$ which are sufficiently close to each other $p$-adically, by Krasner's lemma. Further tricks are available to ease the task, e.g., over $\mathbb{Q}_7(\sqrt{7})$, one may perform a change of variable $X\mapsto \frac{X}{\sqrt{7}u}$ and consider the polynomial $7u^2F(u,X/\sqrt{7}u) = X^7(X+7\sqrt{7}u) - (X^2+\sqrt{7}uX+49u^2)$, which, as already a low-precision calculation verifies, splits into distinct linear factors over $\mathbb{Q}_7(\sqrt{7})$.}

Note furthermore that, since no odd integer is a branch point of the $PGL_2(7)$-extension $E(u)/\mathbb{Q}(u)$, all specializations at such values are Galois extensions of $\mathbb{Q}$ with group a subgroup of $PGL_2(7)$. Since the Galois group is generated by the set of all inertia subgroups, we obtain a subgroup of $PGL_2(7)$ generated by involutions of cycle type $(2^3.1^2)$. Furthermore, the polynomial $F(1,X)$ is irreducible (of degree $8$) modulo $2$, whence the Galois group for all specializations at odd integers must contain an $8$-cycle (namely, the Frobenius at $2$)! One may now verify directly that $PGL_2(7)$ has no proper subgroup which is generated by involutions as above while also containing an $8$-cycle. Hence $Gal(E(u)_{u_0}/\mathbb{Q})=PGL_2(7)$ without exception. 
 Propositions \ref{lem:prime_kernel} and \ref{prop:serre} now yield that $(E(u))_{u_0}/\mathbb{Q}$ embeds into an extension $L/\mathbb{Q}$ with group $SL_2(7).2$, and if $F$ denotes the fixed field of $SL_2(7)$ in $L$, then $L/F$ can be chosen everywhere unramified. It remains to calculate the quadratic field $F$, which is just given by the square-root of the discriminant of $F(u_0,X)$ above. Computation with Magma confirms the shape claimed in the statement of the theorem.
\end{proof}

\begin{remark}
\label{rem:dirprod}
A logical next step would be to try to realize more general central extensions of $PSL_2(7)$ as Galois groups of unramified extensions of quadratic number fields, using, e.g., Theorem \ref{thm:q7}. This is easy for the direct product $SL_2(7)\times C_2$. Indeed, the values $-7\cdot (108 + 7^{10}t^6)$ featuring in the assertion of Theorem \ref{thm:q7} are all strictly divisible by $7$, and the factor ($108+7^{10}t^6$) is not a square (due to being too close to another square). Hence, the extension $\mathbb{Q}(\sqrt{-7\cdot (108 + 7^{10}t^6)}, \sqrt{-7})/\mathbb{Q}(\sqrt{-7\cdot (108 + 7^{10}t^6)})$ is an unramified quadratic extension (necessarily linearly disjoint from the $SL_2(7)$-extension constructed above, which has no quadratic subextensions). In comparison, the same question for the group $SL_2(7)\times C_3$ already seems challenging: quadratic number fields with class number divisible by $3$ can be parameterized effectively, but requiring them to simultaneously have the shape of Theorem \ref{thm:q7} leads to difficult diophantine conditions. E.g., it is easy to exhibit infinitely many quadratic number fields of the shape $\mathbb{Q}(\sqrt{-u(4u+27)})$ with class number divisible by $3$ (for a more general result, see \cite[Main Theorem]{KiMi}). 
Demanding those to simultaneously fall into the shape of Theorem \ref{thm:q7} amounts to finding (suitable) integral points on the conic $U(4U+27) - 7D V^2=0$, with $D=108 + 7^{10}t^6$ for some odd integer $t$, as in the Theorem. I do not know whether such points exist for infinitely many $D$ of the prescribed shape.\end{remark}

{\bf Acknowledgement}: I would like to thank Kwang-Seob Kim for many helpful discussions.

\end{document}